\documentclass[12pt]{article}

\usepackage{amsfonts, amsmath, amssymb, amsthm}
\usepackage{a4}

\DeclareMathOperator{\sn}{sn}

\newtheorem{thm}{Theorem}
\newtheorem*{stm}{Statement}
\newtheorem*{xpr}{Experimental result}
\theoremstyle{remark}
\newtheorem*{rmk}{Remark}

\author{Igor G. Korepanov}
\title{Special 2-cocycles and 3--3 Pachner move relations in Grassmann algebra}
\date{\normalsize January 2013;\quad Appendix --- July 2013}

\begin{document}

\maketitle

\begin{abstract}
Grassmann-algebraic relations, corresponding naturally to Pachner move 3--3 in four-dimensional topology, are presented. They involve 2-cocycles of two specific forms, and some more homological objects.
\end{abstract}

\section{Introduction}\label{s:i}

This short note presents a construction of a 3--3 Pachner move relation in Grassmann algebra --- the four-dimensional analogue of pentagon relation known in three-dimensional algebraic topology. Below:
\begin{itemize}\itemsep 0pt
\item in Section~\ref{s:GB}, the Grassmann--Berezin calculus of anticommuting variables is briefly recalled,
\item in Section~\ref{s:33}, Pachner move 3--3 is explained, together with a possible form of algebraic relation corresponding to it,
\item in Section~\ref{s:W}, actual four-simplex Grassmann weights are presented satisfying this relation, and a cocyclic property of their coefficients is stated,
\item in Section~\ref{s:d}, a generalization of our Grassmann weights is proposed involving new coefficients, also, apparently, of homological nature,
\item and in Appendix on page~\pageref{a:rank4}, different Grassmann weights are presented, whose parameterization uses Jacobi elliptic functions. According to numerical evidence, these Grassmann weights also satisfy the same 3--3 relation.
\end{itemize}

\section{Grassmann--Berezin calculus}\label{s:GB}

A \emph{Grassmann algebra} over a field~$\mathbb F$ of characteristic $>2$ is an associative $\mathbb F$-algebra with unity, generators~$x_i$ and relations
\begin{equation*}
x_i x_j = -x_j x_i .
\end{equation*}
In particular, $x_i^2 =0$, so an element of a Grassmann algebra is a polynomial of degree $\le 1$ in each~$x_i$. If it consists only of monomials of even (odd) total degree, it is called and even (odd) element.

The \emph{exponent} is defined by its Taylor series. For instance,
\[
\exp (x_1x_2+x_3x_4) = 1+x_1x_2+x_3x_4+x_1x_2x_3x_4 .
\]

The \emph{Berezin integral}~\cite{B} in a variable (${}={}$generator)~$x_i$ is, by definition, the $\mathbb F$-linear operator
\[
f \mapsto \int f \, \mathrm dx_i
\]
in Grassmann algebra satisfying
\begin{equation*}
\int \mathrm dx_i =0, \quad \int x_i\, \mathrm dx_i =1, \quad \int gh\, \mathrm dx_i = g \int h\, \mathrm dx_i,
\end{equation*}
if $g$ does not contain~$x_i$; multiple integral is understood as iterated one, according to the following model:
\begin{equation*}
\iint xy\, \mathrm dy\, \mathrm dx = \int x \left( \int y\, \mathrm dy \right) \mathrm dx = 1 .
\end{equation*}

\section{Pachner move 3--3 and a proposed form of algebraic relation}\label{s:33}

Pachner moves~\cite{Pachner} are elementary local rebuildings of a manifold triangulation. A triangulation of a piecewise-linear manifold can be transformed into another triangulation using a finite sequence of Pachner moves, see~\cite{Lickorish} for a pedagogical introduction.

There are five (types of) Pachner moves in four dimensions, of which move 3--3 is, in some informal sense, central. It transforms a cluster of three four-simplices situated around a two-face into a cluster of three other four-simplices situated around another two-face, and occupying the same place in the manifold. We say that these clusters form the left- and right-hand sides of Pachner move, respectively. There are six vertices in each cluster, we denote them $1\dots 6$, and the four-simplices will be 12345, 12346 and~12356 in the l.h.s., and 12456, 13456 and~23456 in the r.h.s. Thus, the common inner two-face is 123 in the l.h.s., and 456 in the r.h.s.

An algebraic relation whose l.h.s.\ and r.h.s.\ can be said to correspond naturally to the l.h.s.\ and r.h.s.\ of a Pachner move gives hope of constructing an invariant of piecewise-linear manifolds.

\begin{rmk}
Four other Pachner moves in four-dimensional topology are $2\leftrightarrow 4$ and $1\leftrightarrow 5$. Experience shows that if an interesting formula related to move 3--3 has been discovered, then there are also formulas corresponding to other moves.
\end{rmk}

The Grassmann-algebraic Pachner move relations proposed in this paper have the following form:
\begin{multline}\label{33}
f_{123} \int \mathcal W_{12345} \mathcal W_{12346} \mathcal W_{12356} \,\mathrm dx_{1234} \,\mathrm dx_{1235} \,\mathrm dx_{1236} \\
 = \pm f_{456} \iiint \mathcal W_{12456} \mathcal W_{13456} \mathcal W_{23456} \,\mathrm dx_{1456} \,\mathrm dx_{2456} \,\mathrm dx_{3456}.
\end{multline}
Here Grassmann variables~$x_{ijkl}$ are attached to all three-faces~$ijkl$; the \emph{Grassmann weight}~$\mathcal W_{ijklm}$ of a four-simplex~$ijklm$ depends on (i.e., contains) the variables on its three-faces, e.g., $\mathcal W_{12345}$ depends on $x_{1234}$, $x_{1235}$, $x_{1245}$, $x_{1345}$ and~$x_{2345}$. The integration goes in variables on \emph{inner} three-faces in the corresponding side of Pachner move, while the result depends on the variables on boundary faces. Also, there are numeric factors~$f_{ijk}$ before the integrals, thought of as attached to the respective inner two-faces $ijk=123$ or~$456$.

\section{Grassmann weights satisfying the 3--3 relation, and a cocyclic property of coefficients}\label{s:W}

We now present Grassmann four-simplex weights~$\mathcal W_{ijklm}$ and factors~$f_{ijk}$, satisfying the 3--3 algebraic relation~\eqref{33}. These will depend on the \emph{coordinates} of vertices: we attach to each vertex~$i$ two numbers $\xi_i,\eta_i\in \mathbb F$ that must be generic enough so that the expressions~\eqref{phi} below never vanish.

\begin{rmk}
Or we can take \emph{indeterminates} over~$\mathbb F$ --- algebraically independent elements --- for $\xi_i$ and~$\eta_i$.
\end{rmk}

We define~$\mathcal W_{ijklm}$ as the following Grassmann--Gaussian exponent:
\begin{equation}\label{exp}
\mathcal W_{ijklm} = \exp \Phi_{ijklm},
\end{equation}
where
\begin{equation}\label{Phi}
\Phi_{ijklm} = p_{ijklm} \sum_{\substack{\text{over 2-faces }abc\\[.3ex] \text{ of }ijklm}} \epsilon_{d_1abcd_2}^{ijklm} \, \varphi_{abc} \, x_{\{abcd_1\}} x_{\{abcd_2\}},
\end{equation}
and below we explain the notations in~\eqref{Phi}.

First, both $p_{ijklm}$ and~$\epsilon_{d_1abcd_2}^{ijklm}$ are signs. The first of them reflects the consistent orientation of four-simplices, namely, for the left-hand side
\[
p_{12345}=1,\quad p_{12346}=-1,\quad p_{23456}=1,
\]
and for the right-hand side
\[
p_{12456}=1,\quad p_{13456}=-1,\quad p_{23456}=1.
\]
As for the epsilon, it is the sign of permutation between the sequences of its subscripts and superscripts.

Second, the value~$\varphi_{abc}$ is defined as follows:
\begin{equation}\label{phi}
\varphi_{abc}=\left| \begin{matrix} 1 & 1 & 1 \\ \xi_a & \xi_b & \xi_c \\ \eta_a & \eta_b & \eta_c \end{matrix} \right|.
\end{equation}
Thus, $\varphi_{abc}$ belongs to an \emph{oriented} two-face: for instance, $\varphi_{abc}=-\varphi_{bac}$.

And third, the curly brackets in~\eqref{Phi} serve to emphasize that the Grassmann variable~$x_{\{abcd\}}$ does \emph{not} depend on the order of indices $a,b,c,d$.

\begin{thm}\label{th:f}
The weights~$\mathcal W_{ijklm}$ defined in this Section satisfy the relation~\eqref{33}, with
\[
f_{ijk}=\frac{1}{\varphi_{ijk}},
\]
and the sign before the right-hand side is minus.
\end{thm}

\begin{proof}
Direct calculation. I used our package~PL~\cite{PL} for manipulations in Grassmann algebra.
\end{proof}

\begin{stm}
Values~$\varphi_{abc}$ form a \emph{2-cocycle}: for a tetrahedron~$abcd$,
\[
\varphi_{bcd}-\varphi_{acd}+\varphi_{abd}-\varphi_{abc}=0.
\]
\end{stm}

\begin{proof}
Simple calculation using~\eqref{phi}.
\end{proof}

\section{A generalization involving still more homological objects}\label{s:d}

We are now going to generalize our Grassmann four-simplex weights using still more objects of, apparently, homological nature.

Calculations show that the exponent~\eqref{exp} has actually no terms of degree${}>2$, that is,
\begin{equation}\label{1+Phi}
\exp \Phi_{ijklm} = 1 + \Phi_{ijklm} .
\end{equation}
We now change the definition~\eqref{exp} to the following:
\begin{equation}\label{h}
\mathcal W_{ijklm} = h_{ijklm} + \Phi_{ijklm},
\end{equation}
where $h_{ijklm}$ is some numeric coefficient (or, more generally, an even element of Grassmann algebra).

\begin{thm}
The relation~\eqref{33} holds also for weights defined according to~\eqref{h}, provided the coefficients~$h_{ijklm}$ in its r.h.s.\ are expressed through those in its l.h.s. as follows:
\begin{eqnarray*}
\varphi_{456}\,h_{12456} &=& \varphi_{345}\,h_{12345} - \varphi_{346}\,h_{12346} + \varphi_{356}\,h_{12356}\,,\\
\varphi_{456}\,h_{13456} &=& \varphi_{245}\,h_{12345} - \varphi_{246}\,h_{12346} + \varphi_{256}\,h_{12356}\,,\\
\varphi_{456}\,h_{23456} &=& \varphi_{145}\,h_{12345} - \varphi_{146}\,h_{12346} + \varphi_{156}\,h_{12356}\,.
\end{eqnarray*}
\end{thm}

\begin{proof}
Direct calculation.
\end{proof}

There appears to be analogy between the constructions in this paper and those in~\cite{exotic-second}. Recall that second homologies, in their exotic form, do enter in Grassmanian 3--3 relations in the mentioned paper. Our present relations look substantially simpler, which may be important for constructing a TQFT. Their homological nature is still to be clarified.

The 3--3 relations proposed here have been found while trying to generalize the pentagon relation in~\cite{KS} to four-dimensional case.

\subsection*{Acknowledgements}

The calculations in this paper have been done using GAP~\cite{GAP} and Maxima~\cite{maxima}. I would like to thank the creators and maintainers of these systems.

\appendix
\setcounter{secnumdepth}{0}

\section{Appendix: A relation with a quadratic form of rank 4}\label{a:rank4}

The explanation of relation~\eqref{1+Phi} for the Grassmanian quadratic form~$\Phi_{ijklm}$ introduced in Section~\ref{s:W} lies in the fact that $\Phi_{ijklm}$ has \emph{rank~2}. A generic Grassmanian quadratic form of five variables has, however, rank~4 --- the maximal rank of an antisymmetric $5\times 5$ matrix. It is natural to expect (for instance, from a comparison with the pentagon relation in~\cite{KS}) that relations with quadratic forms of rank~4 will expose richer mathematical structure than those of rank~2.

In this Appendix, I report new Grassmann weights~$\mathcal W_{ijklm}$, with $\Phi_{ijklm}$ of rank~4. They are constructed almost exactly as in Section~\ref{s:W}, except that we specify the field~$\mathbb F$ to be the field of complex numbers, $\mathbb F = \mathbb C$, and $\varphi_{abc}$ is defined not by~\eqref{phi}, but as follows:
\begin{equation}\label{ell}
\varphi_{abc} = \sn (\alpha_a-\alpha_b) \sn (\alpha_b-\alpha_c) \sn (\alpha_c-\alpha_a)  \,.
\end{equation}
Here $\sn(\,\cdot\,) = \sn(\,\cdot\,,k)$ is the Jacobi elliptic sine of some fixed modulus~$k$, and there is just one complex number~$\alpha_i$ at each vertex~$i$.

\begin{xpr}
The weights~$\mathcal W_{ijklm}$ with $\varphi_{abc}$ defined according to~\eqref{ell} satisfy the same relation~\eqref{33}, refined according to the same Theorem~\ref{th:f}.
\end{xpr}

\noindent
Also, it is not hard to show that the values~\eqref{ell} have the same cocyclic property as described in the Statement after Theorem~\ref{th:f}.

\smallskip

These new Grassmann weights were found by guess-and-try method combined with some theoretical ideas that are to be disclosed later. By now, I was only able to check numerically the validity of relation~\eqref{33} for these weights.


\small
\begin{thebibliography}{99}

\bibitem{B}
F.A. Berezin,
The Method of Second Quantization (in Russian), Nauka, Moscow, 1965; English transl.: Academic Press,
New York, 1966.

\bibitem{GAP}
GAP -- Groups, Algorithms, Programming -- a System for Computational Discrete Algebra, http://www.gap-system.org

\bibitem{PL}
A.I. Korepanov, I.G. Korepanov and N.M. Sadykov, \ 
PL: Piecewise-linear topology using GAP,
http://sf.net/projects/plgap/

\bibitem{exotic-second}
I.G. Korepanov,
Deformation of a $3 \to 3$ Pachner move relation capturing exotic second homologies, arXiv:1201.4762.

\bibitem{KS}
I.G. Korepanov and N.M. Sadykov,
Pentagon relations in direct sums and Grassmann algebras, arXiv:1212.4462.

\bibitem{Lickorish}
W.B.R. Lickorish,
Simplicial moves on complexes and manifolds,
Geom. Topol. Monogr. {\bf 2} (1999), 299--320, arXiv:math/9911256.

\bibitem{maxima}
Maxima, a Computer Algebra System, http://maxima.sourceforge.net

\bibitem{Pachner}
U. Pachner,
PL homeomorphic manifolds are equivalent by elementary shellings,
Europ. J. Combinatorics {\bf 12} (1991), 129--145.

\end{thebibliography}



\normalsize
\end{document}